\documentclass[11pt]{article}

\usepackage[T1]{fontenc}
\usepackage[utf8]{inputenc}
\usepackage{amsmath,amssymb,amsfonts,amsthm,mathtools}
\usepackage{graphicx}
\usepackage{subcaption}
\usepackage{float}
\usepackage[dvipsnames]{xcolor}
\usepackage[hidelinks]{hyperref}
\usepackage[a4paper, left=2cm, right=2cm, top=2.5cm, bottom=2.5cm]{geometry}

\newcommand{\gscat}{\ensuremath{\mathrm{gscat}}}

\newcommand{\join}{\mathbin{*}}
\newcommand{\arb}{\ensuremath{\mathrm{arb}}}

\newtheorem{lemma}{Lemma}
\newtheorem{definition}[lemma]{Definition}

\newtheorem{theorem}[lemma]{Theorem}

\newtheorem{corollary}[lemma]{Corollary}
\newtheorem{example}[lemma]{Example}

\title{Arboricity and Simplicial Geometric Category of Wedges and Joins of Graphs}

\author{
\textsc{Nursultan Kuanyshov}\textsuperscript{1}
and
\textsc{Islam Yeginbay}\textsuperscript{2}
}
\date{}

\begin{document}
\maketitle

\begin{abstract}
We investigate the behavior of arboricity under two fundamental graph operations, namely wedges and joins, proving an exact formula for wedges and establishing general upper and lower bounds for joins. Using the characterization of the simplicial geometric category of connected graphs in terms of arboricity, we derive a wedge formula for simplicial geometric category and obtain corresponding estimates for graph joins. Finally, we illustrate these results through explicit computations for several classes of graphs by constructing forest decompositions and the associated covers by strongly collapsible subcomplexes.
\end{abstract}


\section{Introduction}

The Lusternik--Schnirelmann category~\cite{LS,CLOT} and its geometric version introduced by Fox~\cite{Fox1941} are classical numerical invariants measuring the topological complexity of spaces. In recent years, several combinatorial analogues have been developed for finite topological spaces and simplicial complexes~\cite{FMV2015,FTMVMV2}. Among them, the simplicial geometric category introduced by Fernández-Ternero, Macías-Virgós and Vilches~\cite{FMV2015} is defined using the notion of strong collapse introduced by Barmak and Minian~\cite{BarmakMinian2012} and provides a discrete counterpart of Fox's geometric category.

A remarkable feature of the one-dimensional case is that simplicial geometric category admits a purely graph-theoretic description. Fernández-Ternero, Macías-Virgós, Minuz and Vilches proved that the simplicial geometric category of every connected graph coincides with its arboricity~\cite{FTMVMV2}. Consequently, questions concerning simplicial geometric category may be translated into questions on arboricity. This correspondence provides an effective method for studying simplicial geometric category through classical techniques from graph theory, particularly the theory of decompositions of graphs into forests~\cite{NashWilliams1964,BondyMurty}.

The purpose of this paper is to exploit this connection. Rather than studying simplicial geometric category directly, we first establish new results on arboricity under natural graph operations and then transfer these results to simplicial geometric category.

Our first main theorem determines the arboricity of wedges:$$\arb(K\vee L)=\max\{\arb(K),\arb(L)\}.$$
Using the equality between arboricity and simplicial geometric category for connected graphs~\cite{FTMVMV2}, we immediately obtain the corresponding wedge formula for simplicial geometric category. While related results are known for simplicial Lusternik--Schnirelmann category~\cite{FMV2015,FTMVMV2}, the geometric version appears not to have been previously established.

Our second main result concerns graph joins. We derive general lower and upper bounds for the arboricity of the join of two graphs. These estimates depend on the arboricities of the factors together with the contribution of the complete bipartite graph introduced by the join. As an application of the correspondence between arboricity and simplicial geometric category~\cite{FTMVMV2}, we obtain corresponding estimates for the simplicial geometric category of graph joins.

Finally, we demonstrate the effectiveness of our approach through explicit computations for several important graph families, including bouquets of cycles, cactus graphs, complete bipartite graphs, fan graphs, wheel graphs, and joins of paths with independent sets. In each case, the arboricity is realized by explicit decompositions into forests, which translate naturally into explicit covers by strongly collapsible subcomplexes.

\section{Preliminaries}

This section fixes the conventions and recalls the results used throughout the
paper. All simplicial complexes and graphs are finite, and all graphs are
simple. A graph is regarded as a one-dimensional simplicial complex whenever
$\gscat$ is applied to it.

\subsection{Simplicial Geometric Category}

\begin{definition}
An \emph{abstract simplicial complex} $K$ on a vertex set $V(K)$ is a collection of finite nonempty subsets of $V(K)$ such that, whenever $\sigma\in K$ and $\varnothing\neq\tau\subseteq\sigma$, we have $\tau\in K$. The elements of $K$ are called \emph{simplices}, and the dimension of a simplex $\sigma$ is defined by $$ \dim \sigma = |\sigma|-1. $$ A simplex is \emph{maximal} if it is not properly contained in any other simplex of $K$. The dimension of $K$ is $$\dim K=\max\{\dim\sigma:\sigma\in K\}.$$
\end{definition}

We first recall the notion of strong collapse introduced by Barmak and
Minian~\cite{BarmakMinian2012}.

\begin{definition}
Let $K$ be a simplicial complex and let $u,v\in V(K)$ with $u\neq v$.
The vertex $v$ is said to be \emph{dominated} by $u$ if every maximal
simplex of $K$ containing $v$ also contains $u$. Equivalently, $$\operatorname{lk}_K(v)\subseteq \operatorname{st}_K(u).$$
In this case, the deletion of $v$, together with all simplices containing
it, is called an \emph{elementary strong collapse}. A finite sequence of
elementary strong collapses is called a \emph{strong collapse}. The complex
$K$ is \emph{strongly collapsible} if it strongly collapses to a single
vertex.
\end{definition}

\begin{definition}~\cite{FMV2015}
The \emph{simplicial geometric category} of a finite
simplicial complex $K$, denoted by $\gscat(K)$, is the least integer $m\geq0$ such that $$K=C_0\cup C_1\cup\cdots\cup C_m,$$
where each $C_i$ is strongly collapsible. Notice that strongly collapsible subcomplexes must be connected.
\end{definition}

\subsection{Graphs}

Following Bondy and Murty~\cite{BondyMurty}, we recall the standard graph-theoretic notation used throughout the paper. A finite simple graph is a pair $G=(V(G),E(G))$, where $V(G)$ is a finite set of vertices and $E(G)\subseteq\bigl\{\{u,v\}:u,v\in V(G),\ u\neq v\bigr\}$ is its set of edges. As usual, we write $uv$ for the unordered pair $\{u,v\}$. Two vertices $u,v\in V(G)$ are adjacent if $uv\in E(G)$. For a subset $S\subseteq V(G)$, the induced subgraph $G[S]$ has vertex set $S$ and edge set $E(G[S])=\{uv\in E(G):u,v\in S\}$. We use the customary notation for the following standard families of graphs. The complete graph $K_n$ is the graph on $n$ vertices in which every pair of distinct vertices is adjacent, whereas the edgeless graph $\overline{K}_n$ has $n$ vertices and no edges. The complete bipartite graph $K_{m,n}$ has a vertex partition $V(K_{m,n})=X\sqcup Y$, where $|X|=m$ and $|Y|=n$, and its edges are precisely all pairs $xy$ with $x\in X$ and $y\in Y$. The path graph $P_n$ has vertices $v_1,\ldots,v_n$ and edge set $E(P_n)=\{v_iv_{i+1}:1\leq i\leq n-1\}$. For $n\geq3$, the cycle graph $C_n$ is obtained from $P_n$ by adding the edge $v_nv_1$.

\begin{definition}
A graph is a \emph{forest} if it contains no cycles. Equivalently, each connected component of a forest is a tree.
\end{definition}

\begin{definition}
A cactus graph is a connected graph in which every edge belongs to
at most one cycle. Equivalently, any two distinct cycles have at most one vertex in common.
\end{definition}

\begin{definition}[Induced subgraph {\cite{BondyMurty}}]
Let $G=(V,E)$ be a graph and let $S\subseteq V$. The subgraph of $G$ induced by $S$, denoted by $G[S]$, is the graph with vertex set $S$ and edge set $$E(G[S])=\{uv\in E\mid u,v\in S\}.$$ In other words, $G[S]$ contains all edges of $G$ whose endpoints both belong to $S$.
\end{definition}

\begin{definition}
Let $G$ be a graph. If $E(G)\neq\varnothing$, let $k$ be the minimum number of forests whose edge sets partition $E(G)$. The
(reduced) \emph{arboricity} of $G$ is defined by $\arb(G)=k-1.$ 
\end{definition}

By our convention, $\arb(G)$ is one less than the classical arboricity of $G$ \cite{NashWilliams1964}. If $E(G)=\varnothing$, we set $\arb(G)=0$. 

\begin{theorem}[Nash--Williams~\cite{NashWilliams1964}]
\label{thm:nw} For every graph $G$ containing at least one edge, $$\arb(G)=\max\left\{0, \,\max_{\substack{H\subseteq G\\|V(H)|\geq2}}\left\lceil\frac{|E(H)|}{|V(H)|-1}\right\rceil-1\right\}.$$ 
\end{theorem}

\begin{definition}[Wedge]
Let $K$ and $L$ be graphs sharing exactly one vertex $v$ and otherwise
disjoint. Their \emph{wedges}, denoted by $K\vee_vL$, is their union after
the two chosen copies of $v$ have been identified.
\end{definition}

\begin{definition}[Join]
Let $G_1=(V_1,E_1)$ and $G_2=(V_2,E_2)$ have disjoint vertex sets.
The \emph{join} $G_1\join G_2$ is the graph with
$$E(G_1\join G_2)=E_1\cup E_2\cup\{uv:u\in V_1,\ v\in V_2\}.$$
\end{definition}

\begin{lemma}[Monotonicity]\label{lem:mono}
If $H\subseteq G$, then $$\arb(H)\leq\arb(G).$$
\end{lemma}

\begin{proof}
Restrict any forest decomposition of $G$ to $H$.
\end{proof}

\begin{lemma}[Edge bound]\label{lem:edgebound}
Let $G$ be a graph with $|V(G)|=n\geq 1$. If $\arb(G)\leq t$, where $\arb$ denotes the reduced arboricity, then $$|E(G)|\leq (t+1)(n-1).$$
\end{lemma}

\begin{proof}
Since $\arb(G)\leq t$, the edge set of $G$ can be partitioned into at most $t+1$ forests. Each forest on $n$ vertices has at most $n-1$ edges. Therefore, $$|E(G)|\leq (t+1)(n-1).$$
\end{proof}

\begin{lemma}[Bipartite density]
\label{lem:bipdensity}
For positive integers $s,t$, define $$g(s,t)=\frac{st}{s+t-1}.$$ Then $g$ is nondecreasing in each variable.
\end{lemma}

\begin{proof}
Fix $t\geq1$. For every $s\geq1$, we have

$$
\begin{aligned}
g(s+1,t)-g(s,t)
&=
\frac{(s+1)t}{s+t}
-
\frac{st}{s+t-1}= \\
&=
\frac{t(t-1)}{(s+t)(s+t-1)}
\geq0.
\end{aligned}
$$

Thus, $g(s,t)$ is nondecreasing in $s$. By symmetry, it is also
nondecreasing in $t$.
\end{proof}

\begin{corollary}
\label{cor:kmn}
For positive integers $m,n$, $$\arb(K_{m,n})=\left\lceil\frac{mn}{m+n-1}\right\rceil-1.$$
\end{corollary}

\begin{proof}
Taking the whole vertex set of $K_{m,n}$ in
Theorem~\ref{thm:nw} gives $$\arb(K_{m,n})\geq\left\lceil
\frac{|E(K_{m,n})|}{|V(K_{m,n})|-1}\right\rceil-1=\left\lceil\frac{mn}{m+n-1}\right\rceil-1.$$

For the reverse inequality, let $S\subseteq V(K_{m,n})$ with
$|S|\geq2$. Suppose that $S$ contains $s$ vertices from the first
part and $t$ vertices from the second part. Then $|E(K_{m,n}[S])|=st$ and $|S|-1=s+t-1.$ If $s=0$ or $t=0$, then $K_{m,n}[S]$ has no edges, and hence$$\frac{|E(K_{m,n}[S])|}{|S|-1}=0.$$
If $s,t\geq1$, then $s\leq m$ and $t\leq n$, so Lemma~\ref{lem:bipdensity} gives $$ \frac{|E(K_{m,n}[S])|}{|S|-1}=\frac{st}{s+t-1}\leq\frac{mn}{m+n-1}.$$
Therefore, for every $S\subseteq V(K_{m,n})$ with $|S|\geq2$,
$$\left\lceil\frac{|E(K_{m,n}[S])|}{|S|-1}\right\rceil\leq \left\lceil\frac{mn}{m+n-1}\right\rceil.$$Applying Theorem~\ref{thm:nw}, we obtain$$\arb(K_{m,n})\leq\left\lceil\frac{mn}{m+n-1}\right\rceil-1.$$
Combining the two inequalities proves the result.
\end{proof}

\begin{corollary}
\label{cor:complete}
For every $n\geq2$, $$ \arb(K_n)=\left\lceil\frac{n}{2}\right\rceil-1. $$
\end{corollary}

\begin{proof}
An induced copy of $K_s$ has density $$ \frac{\binom{s}{2}}{s-1}=\frac{s}{2},$$
which is maximized at $s=n$. We apply Theorem~\ref{thm:nw} to get the statement of Corollary \ref{cor:complete}.
\end{proof}

\begin{theorem}[D. Fern\'andez-Ternero, E. Mac\'ias-Virg\'os, E. Minuz and J.A. Vilches ~\cite{FTMVMV2}]
\label{thm:graph-gscat}
Let $G$ be a connected graph containing at least one edge. Then $$ \gscat(G)=\arb(G). $$
\end{theorem}

\begin{proof}
Briefly, a strongly collapsible one-dimensional complex is a tree. A decomposition of $G$ into $a=\arb(G)+1$ forests can be enlarged, inside the connected graph $G$, to a cover by $a$ trees;
hence $\gscat(G)\leq a-1=\arb(G)$. Conversely, from a cover by $m+1$ trees, assign
each edge to one tree containing it. The assigned edge sets are forests and
partition $E(G)$, so $\arb(G)\leq m$.
\end{proof}

\section{Main Results}

\subsection{Arboricity of Wedges}

In this subsection, we prove a theorem that determines the arboricity of the wedge of two graphs in terms of the arboricities of its factors. We then derive the corresponding result for finite iterated wedges and illustrate these formulas with several examples.

\begin{theorem}
\label{thm:wedge} Let $G=K\vee_v L$ be the wedges of $K$ and $L$ at $v$. Then $$\arb(G)=\max\{\arb(K),\arb(L)\}.$$
\end{theorem}

\begin{proof}
\emph{Lower bound.} Since $K,L\subseteq G$ as subgraphs, $\arb(G)\ge\arb(K)$ and $\arb(G)\ge\arb(L)$ by Lemma~\ref{lem:mono}. Hence, we obtain that $\arb(G)\ge\max\{\arb(K),\arb(L)\}.$

\emph{Upper bound.} Let $H\subseteq G$, $|V(H)|\ge 2$, attain the maximum in Theorem~\ref{thm:nw} for $\arb(G)$, and set $H_K:=H\cap K$, $H_L:=H\cap L$. We first consider the case when $H$ is contained entirely in one of the two graphs. If $H\subseteq K$ (or $H\subseteq L$), then trivially inequalities hold: $\lceil|E(H)|/(|V(H)|-1)\rceil-1\le\arb(K)\le\max\{\arb(K),\arb(L)\}$.

It remains to consider the case when $H$ meets both $K\setminus\{v\}$ and $L\setminus\{v\}$. Since $v$ is the only vertex joining $K$ and $L$ in $G$, any path in $H$ from $K\setminus\{v\}$ to $L\setminus\{v\}$ must pass through $v$. Therefore, $v\in V(H)$. Moreover, since $E(K)\cap E(L)=\emptyset$, the edges of $H$ split into the edges lying in $H_K$ and those lying in $H_L$. Consequently, $$|E(H)|=|E(H_K)|+|E(H_L)|,\qquad |V(H)|=|V(H_K)|+|V(H_L)|-1,$$
where the second identity follows from the fact that $v$ is the only shared vertex of $H_K$ and $H_L$.

Now, writing $p=|E(H_K)|$, $q=|E(H_L)|$, $r=|V(H_K)|-1$, and $s=|V(H_L)|-1$, we have $|E(H)|=p+q$ and $|V(H)|-1=r+s$. If $r=0$, then $p=0$ as well, and therefore $(p+q)/(r+s)=q/s$. The case $s=0$ is symmetric. Otherwise, when both $r$ and $s$ are positive, the mediant inequality gives $$ \frac{|E(H)|}{|V(H)|-1}=\frac{p+q}{r+s}\le\max\left\{\frac{p}{r},\frac{q}{s}\right\}=$$ $$=\max\left(\frac{|E(H_K)|}{|V(H_K)|-1},\frac{|E(H_L)|}{|V(H_L)|-1}\right)\le\max\{\arb(K)+1,\arb(L)+1\}.$$

Thus, in either case, the density of $H$ does not exceed $\max\{\arb(K)+1,\arb(L)+1\}$. Since $\max\{\arb(K),\arb(L)\}\in\mathbb Z$, taking ceilings and subtracting one gives $\arb(G)\le\max\{\arb(K),\arb(L)\}$.
Combining both bounds proves the theorem.
\end{proof}

An immediate induction gives the following form, which is the most useful one
in examples.

\begin{corollary}
\label{cor:iterated-wedge}
For a finite iterated wedges, $$ \arb\left(\bigvee_{i=1}^{m}G_i\right) =\max_{1\leq i\leq m}\{\arb(G_i)\}. $$
\end{corollary}

\begin{corollary}
\label{cor:iterated-wedge-gscat}
Let $G_1,\ldots,G_m$ be connected graphs containing at least one edge. Then $$\gscat\left(\bigvee_{i=1}^{m}G_i\right)=\max_{1\leq i\leq m}\{\gscat(G_i)\}.$$
\end{corollary}

\begin{proof}
This follows immediately from Corollary~\ref{cor:iterated-wedge} and Theorem~\ref{thm:graph-gscat}, since $\gscat(G)=\arb(G)$ for every connected graph containing at least one edge.
\end{proof}

\subsubsection{Examples and computations}

\begin{example}
\label{ex:bouquet}
Let $B=C_{n_1}\vee C_{n_2}\vee\cdots\vee C_{n_r},$ $n_i\geq3$ be a bouquet of cycles (see Figure ~\ref{fig:bouquet}). For every cycle $C_n$ gives
$\arb(C_n)=1$. Then Corollary~\ref{cor:iterated-wedge} gives $\arb(B)=1.$ Since $B$ is connected, Corollary~\ref{cor:iterated-wedge-gscat} also gives $\gscat(B)=1$. Alternatively, we can cover B with two explicits trees (see Figure ~\ref{fig:bouquet}).  
\end{example}

\begin{figure}[H]
    \centering
    \includegraphics[width=0.5\textwidth]{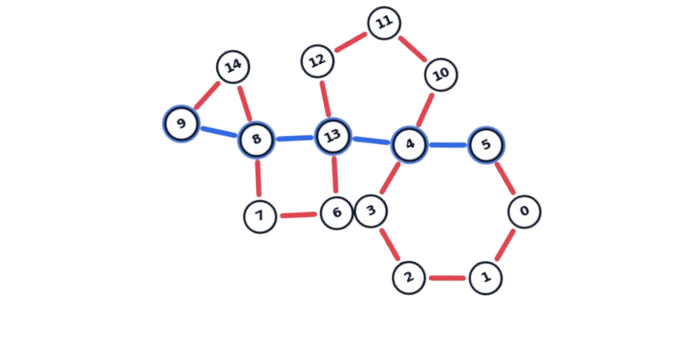}
    \caption{A bouquet of four cycles with $\gscat(B)=1$ explicitly shown two trees in red and blue.}
    \label{fig:bouquet}
\end{figure}

\begin{example}Let G be a Cactus graph (see Figure ~\ref{fig:cactus} (a),(b)). Recall that every connected cactus graph is obtained by successively attaching blocks
that are either single edges or cycles at cut vertices. Since
$\arb(K_2)=0$ and $\arb(C_n)=1$, Corollary~\ref{cor:iterated-wedge} implies
$$
\arb(G)=
\begin{cases}
0, & \text{if $G$ is a tree},\\
1, & \text{if $G$ contains a cycle}.
\end{cases}
$$
\end{example}

\begin{figure}[H]
    \centering
    \begin{subfigure}[b]{0.4\textwidth}
        \centering  
        \includegraphics[width=1\textwidth]{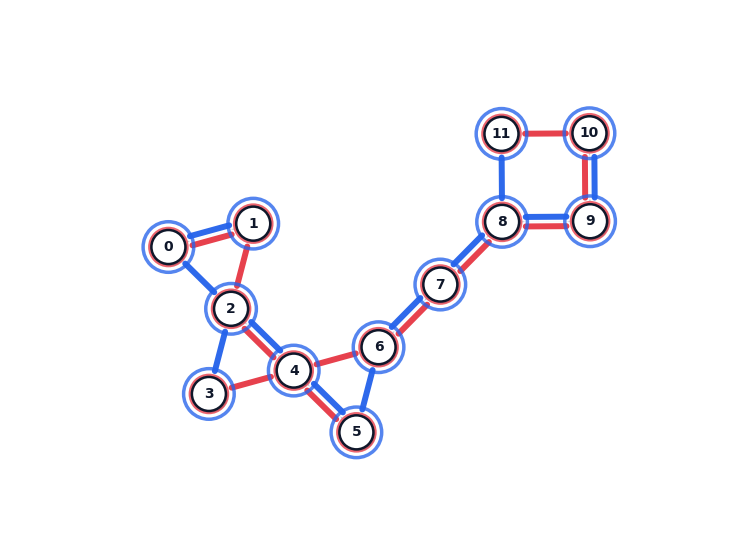}
        \caption{A cactus graph with $\gscat(G)=1$ and an explicit decomposition into two trees, shown in red and blue.}
        \label{fig:cactus_with_cycle}
    \end{subfigure}
    \hfill
    \begin{subfigure}[b]{0.4\textwidth}
        \centering
        \includegraphics[width=1\textwidth]{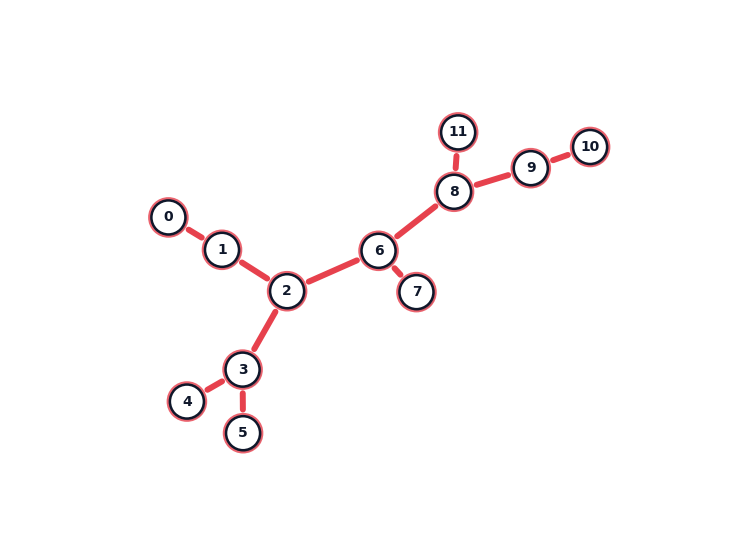}
        \caption{A tree with $\gscat(G)=0$, shown as a single red tree.}
        \label{fig:cactus_as_tree}
    \end{subfigure}
    \caption{Cactus graphs with arboricities $1$ and $0$ and their explicit decompositions into trees.}
    \label{fig:cactus}
\end{figure}

\begin{example}
\label{ex:dense-wedge}
Consider the following graph $G=K_4 \vee K_5$ (see Figure ~\ref{fig:dense-wedge}). By Corollary~\ref{cor:complete},
$\arb(K_5)=2$ and $\arb(K_4)=1$. Hence $\arb(G)=\max\{2,1\}=2$ 
\end{example}

\begin{figure}[H]
    \centering
    \includegraphics[width=0.45\textwidth]{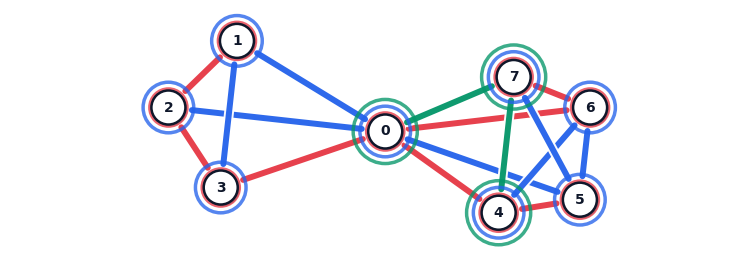}
    \caption{A wedges $K_4\vee K_5$ with $\gscat(K_4 \vee K_5) = 2$ and an explicit decomposition into three trees, shown in red, blue, and green.}
    \label{fig:dense-wedge}
\end{figure}

\subsection{Arboricity of Joins}

In this subsection, we establish lower and upper bounds for the arboricity of the join of two graphs. Using the equality between arboricity and geometric simplicial category for connected graphs, we obtain analogous bounds for $\gscat$. We then illustrate these results with several examples.

\begin{theorem}
\label{thm:join}
Let $G_i=(V_i,E_i)$ be graphs with $|V_i|=n_i\geq 1$, for $i=1,2$, and set $$B=\left\lceil\frac{n_1n_2}{n_1+n_2-1}\right\rceil-1.$$ Then $$\max\{\arb(G_1),\arb(G_2),B\}\leq\arb(G_1\join G_2)\leq\arb(G_1)+\arb(G_2)+B+1.$$
\end{theorem}

\begin{proof}
\emph{Lower bound.}
The subgraph of $G_1\join G_2$ induced by $V_1$ is exactly $G_1$, since the join operation adds no new edges between vertices of $V_1$. Thus, $G_1\subseteq G_1\join G_2$, and Lemma~\ref{lem:mono} gives $\arb(G_1)\leq\arb(G_1\join G_2)$. By the same argument, $\arb(G_2)\leq\arb(G_1\join G_2)$. Moreover, if we retain only the edges joining vertices of $V_1$ to vertices of $V_2$, we obtain the complete bipartite graph $K_{n_1,n_2}$. Hence, $K_{n_1,n_2}\subseteq G_1\join G_2$. Therefore, by Lemma~\ref{lem:mono} and Corollary~\ref{cor:kmn}, $B=\arb(K_{n_1,n_2})\leq\arb(G_1\join G_2)$. Combining these three inequalities, we obtain $$\arb(G_1\join G_2)\geq\max\{\arb(G_1),\arb(G_2),B\}.$$

\emph{Upper bound.}
Let $S\subseteq V_1\cup V_2$ be a subset with $|S|\geq2$, and set $S_1=S\cap V_1$ and $S_2=S\cap V_2$. The edges of the induced subgraph$(G_1\join G_2)[S]$ consist of the edges of $G_1[S_1]$, the edges of $G_2[S_2]$, and all edges joining a vertex of $S_1$ to a vertex of $S_2$. Consequently,
$$|E((G_1\join G_2)[S])|=|E(G_1[S_1])|+|E(G_2[S_2])|+|S_1|\cdot|S_2|.$$

\emph{The first two terms.}
If $S_1=\emptyset$, then $|E(G_1[S_1])|=0$. Otherwise, let $k=|S_1|\geq1$. By Lemma~\ref{lem:mono}, $\arb(G_1[S_1])\leq\arb(G_1)$, and Lemma~\ref{lem:edgebound} therefore gives $|E(G_1[S_1])|\leq(\arb(G_1)+1)(k-1)$. Similarly, $|E(G_2[S_2])|\leq(\arb(G_2)+1)(|S_2|-1)$.

\emph{The cross term.}
Write $s=|S_1|\leq n_1$ and $t=|S_2|\leq n_2$. If $s=0$ or $t=0$, then the cross term is equal to zero. Suppose now that $s,t\geq1$. By Lemma~\ref{lem:bipdensity}, $$g(s,t)\leq g(n_1,t)\leq g(n_1,n_2)=\frac{n_1n_2}{n_1+n_2-1}\leq B+1.$$ Since $g(s,t)=st/(s+t-1)$ and $s+t=|S|$, it follows that
$$|S_1||S_2|=g(s,t)(s+t-1)\leq(B+1)(|S|-1).$$

\emph{Combining the estimates.}
Suppose first that both $S_1$ and $S_2$ are nonempty. Adding the three estimates above, we obtain
$$
\begin{aligned}
|E((G_1\join G_2)[S])|
&\leq(\arb(G_1)+1)(|S_1|-1)+(\arb(G_2)+1)(|S_2|-1)+ \\
&\qquad +(B+1)(|S|-1).\end{aligned}$$
Since$(|S_1|-1)+(|S_2|-1)=|S|-2$, this implies $$|E((G_1\join G_2)[S])|\leq\bigl(\arb(G_1)+\arb(G_2)+B+2\bigr)(|S|-1).$$
If one of $S_1$ or $S_2$ is empty, the same estimate follows directly from Lemma~\ref{lem:mono}. Therefore, Theorem~\ref{thm:nw} gives $\arb(G_1\join G_2)\leq\arb(G_1)+\arb(G_2)+B+1$.
Combining the lower and upper bounds, we conclude that
$$\max\{\arb(G_1),\arb(G_2),B\}\leq\arb(G_1\join G_2)\leq\arb(G_1)+\arb(G_2)+B+1.$$
\end{proof}

\begin{corollary}
\label{cor:join-gscat}
Let $G_i=(V_i,E_i)$ be connected graphs containing at least one edge, with $|V_i|=n_i$, for $i=1,2$, and set
$ B=\left\lceil\frac{n_1n_2}{n_1+n_2-1}\right\rceil-1.$ Then $$\max\{\gscat(G_1),\gscat(G_2),B\}\leq\gscat(G_1\join G_2)\leq\gscat(G_1)+\gscat(G_2)+B+1.$$
\end{corollary}
\begin{proof}
This follows immediately from Theorems~\ref{thm:join}
and~\ref{thm:graph-gscat}.    
\end{proof}

\subsubsection{Examples and computations}

\begin{example}
\label{ex:independent-join}
Let $G=\overline K_m\join\overline K_n=K_{m,n}$, where $m,n\geq1$ (see Figure~\ref{fig:independent-join}). Since $E(\overline K_m)=E(\overline K_n)=\varnothing$, we have $\arb(\overline K_m)=\arb(\overline K_n)=0$. Setting $B=\left\lceil\frac{mn}{m+n-1}\right\rceil-1$, Theorem~\ref{thm:join} gives $B\leq\arb(K_{m,n})\leq B$. Therefore, $\arb(K_{m,n})=\left\lceil\frac{mn}{m+n-1}\right\rceil-1$. Since $K_{m,n}$ is connected, Corollary~\ref{cor:join-gscat} also gives $\gscat(K_{m,n})=\left\lceil\frac{mn}{m+n-1}\right\rceil-1$.
\end{example}

\begin{figure}[H]
    \centering
    \includegraphics[width=0.4\textwidth]{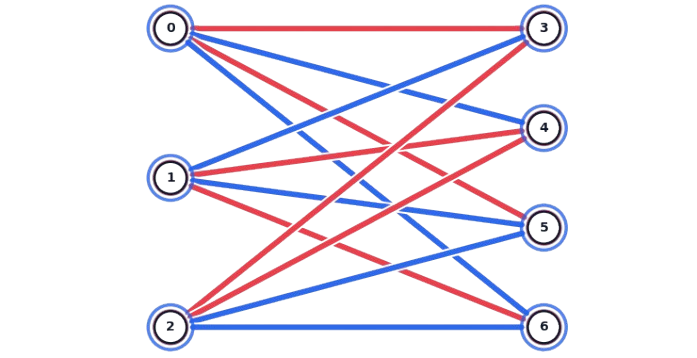}
    \caption{The complete bipartite graph $K_{3,4}$ with $\gscat(K_{3,4}) = 1$ and an explicit decomposition into two spanning trees, shown in red and blue.}
    \label{fig:independent-join}
\end{figure}

\begin{example}
\label{ex:fan}
Let $F_n=K_1\join P_n$, where $P_n$ is the path on $n\geq2$ vertices (see Figure~\ref{fig:fan}). Since $\arb(K_1)=\arb(P_n)=0$ and $B=\left\lceil\frac{n}{n}\right\rceil-1=0$, Theorem~\ref{thm:join} gives $0\leq\arb(F_n)\leq1$. Moreover, $F_n$ contains a triangle formed by the cone vertex and any edge of $P_n$, so $F_n$ is not a forest. Therefore, $\arb(F_n)=1$. Since $F_n$ is connected, Corollary~\ref{cor:join-gscat} also gives $\gscat(F_n)=1$.
\end{example}

\begin{figure}[H]
    \centering
    \includegraphics[width=0.3\textwidth]{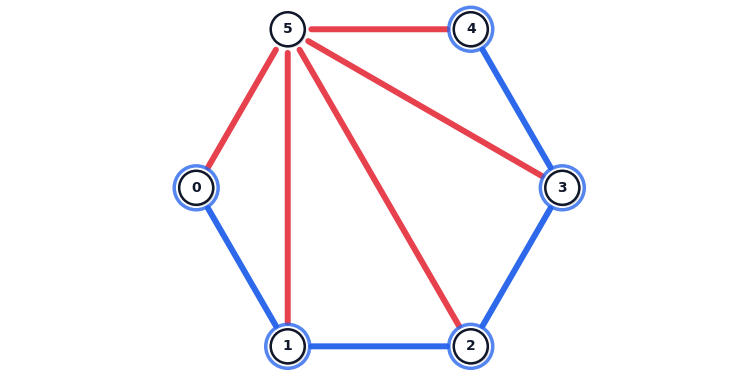}
    \caption{The fan graph $F_5=K_1\join P_5$ with $\gscat(F_5)=1$ and an explicit decomposition into two trees, shown in red and blue.}
    \label{fig:fan}
\end{figure}

\begin{example}
\label{ex:path-independent}
Let $G=P_4\join\overline K_3$ (see Figure~\ref{fig:path-independent}). Since $\arb(P_4)=\arb(\overline K_3)=0$ and $B=\left\lceil\frac{4\cdot3}{4+3-1}\right\rceil-1=1$, Theorem~\ref{thm:join} gives $1\leq\arb(G)\leq2$. Moreover, $G$ has seven vertices and $|E(G)|=|E(P_4)|+4\cdot3=3+12=15$. Applying Theorem~\ref{thm:nw} to the full vertex set gives $\arb(G)\geq\left\lceil\frac{15}{7-1}\right\rceil-1=2$. Therefore, $\arb(P_4\join\overline K_3)=2$. Since $G$ is connected, Corollary~\ref{cor:join-gscat} also gives $\gscat(P_4\join\overline K_3)=2$.
\end{example}

\begin{figure}[H]
    \centering
    \includegraphics[width=0.4\textwidth]{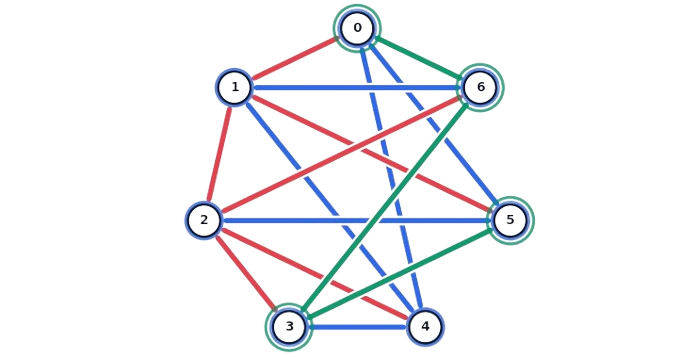}
    \caption{The graph $P_4\join\overline K_3$ with $\gscat(P_4 \join\overline K_3) = 2$ and an explicit decomposition into three trees, shown in red, blue, and green.}
    \label{fig:path-independent}
\end{figure}

\begin{example}
\label{ex:wheel}
Let $W_n=K_1\join C_{n-1}$, where $n\geq4$ (see Figure~\ref{fig:wheel}). Since $\arb(K_1)=0$, $\arb(C_{n-1})=1$, and $B=\left\lceil\frac{n-1}{n-1}\right\rceil-1=0$, Theorem~\ref{thm:join} gives $1\leq\arb(W_n)\leq2$. The upper bound can be sharpened directly. Label the rim vertices by $v_1,\ldots,v_{n-1}$ and the hub by $h$. One forest consists of the rim path $v_1v_2,v_2v_3,\ldots,v_{n-2}v_{n-1}$ together with the spoke $hv_1$, while the second forest consists of the remaining rim edge $v_{n-1}v_1$ and the spokes $hv_2,\ldots,hv_{n-1}$. These two spanning trees contain every edge of $W_n$, so $\arb(W_n)\leq1$. Therefore, $\arb(W_n)=1$. Since $W_n$ is connected, Corollary~\ref{cor:join-gscat} also gives $\gscat(W_n)=1$.
\end{example}

\begin{figure}[H]
    \centering
    \includegraphics[width=0.5\textwidth]{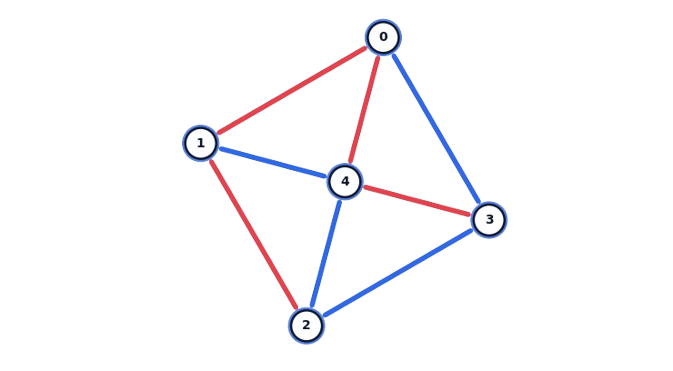}
    \caption{The wheel graph $W_5$ with $\gscat(W_5)=1$ and an explicit decomposition into two spanning trees, shown in red and blue.}
    \label{fig:wheel}
\end{figure}

\section{Conclusion}

We established new results on the arboricity of wedges and joins of graphs and applied them to derive corresponding results for the simplicial geometric category of connected graphs. In particular, we obtained a wedge formula for simplicial geometric category and general estimates for graph joins. These results further illustrate the close relationship between graph-theoretic decompositions and combinatorial topological invariants and provide a useful framework for studying simplicial geometric category through arboricity.

\section*{Acknowledgements}

The work was partially supported by the grant No. AP25796111 of the Science Committee of the Ministry of Science and Higher Education of the Republic of Kazakhstan.

\bigskip

\noindent
\textsuperscript{1,2}
\textsc{SDU University, Kaskelen, Kazakhstan.}

\smallskip

\noindent
\textsuperscript{1}\textit{Email address:}
\texttt{nursultan.kuanyshov@sdu.edu.kz}

\smallskip

\noindent
\textsuperscript{2}\textit{Email address:}
\texttt{islam.yeginbay@gmail.com}

\end{document}